\documentclass[12pt]{amsart}
\usepackage{amssymb}
\usepackage{mathrsfs}
\usepackage{graphicx}
\usepackage{comment}

\newcommand\Q{{\mathbb Q}}
\newcommand\R{{\mathbb R}}
\newcommand\Z{{\mathbb Z}}
\newcommand\N{{\mathbb N}}

\newcommand\eps{\varepsilon}

\newtheorem{theorem}{Theorem}
\newtheorem{lemma}[theorem]{Lemma}

\newtheorem{proposition}[theorem]{Proposition}
\newtheorem{corollary}[theorem]{Corollary}

\theoremstyle{definition}

\theoremstyle{remark}

\title[Subsets closed under addition or subtraction]{Subsets of abelian groups closed under addition or subtraction}

\begin{document}

\author{Art\= uras Dubickas}
	\address{Institute of Mathematics, Faculty of Mathematics and Informatics, Vilnius University, Naugarduko 24,
		LT-03225 Vilnius, Lithuania}
	\email{arturas.dubickas@mif.vu.lt}
	
	\author{Chris Smyth}
	\address{
    School of Mathematics, University of Edinburgh, Edinburgh, EH9 3FD, Scotland, UK}
	\email{c.smyth@ed.ac.uk}

	
	
	\begin{abstract}
		In this article, we first describe all nonempty sets
        of integers $S$ with the property that 
        for all $n$ and $m$ in $S$, not necessarily distinct, the set $\{n-m,n+m\}\cap S$ consists of a single element. These are the sets with at most two elements, one of which is zero, and the infinite sets $\{rk\}$,
        where $r$ is a fixed positive integer and $k$ runs over all integers not divisible by $3$. 
              In the later sections, we solve the analogous problem for subsets of abelian groups. We also discuss, but do not completely solve, the analogous problem for nonabelian groups.
      	\end{abstract}
	
	\maketitle

\section{Introduction} Let $G$ be a given additive abelian group. This paper is concerned with the classification of nonempty subsets $S$ of $G$ with the property that, for all $g,h\in S$, not necessarily distinct, either
\begin{itemize}
\item exactly one of $g+h$ and $g-h$ lies in $S$

or
\item $g+h=g-h\in S$.
\end{itemize}
 We call such a subset $S$ of $G$ a {\it one-subset of $G$}. Thus $S$ is a one-subset of $G$ if and only if
 \[
\#(\{g+h,g-h\}\cap S)=1 \text{ for each $g,h\in S$.}
 \]

 The idea of studying such sets came from root systems, where, within a root system, we noted   the nontrivial conditions on two roots for their sum or difference to also be roots. See for instance \cite[Cor.8.7]{hal}. We decided to look for sets where the possibilities for addition or subtraction within the set were simpler.

 Before giving the classification of one-subsets of general abelian groups (see Section 4), we first consider the special case where the group $G$ is the additive group of integers $\Z$ (see Section 2).  The proof in this case given in Section 3 makes use of the ordering of the integers. The classification proof for general abelian groups is necessarily quite different.

 In   Section \ref{S-ex}, we give some examples of one-subsets for a selection of different abelian groups. In Section 6, we construct some analogous one-subsets of nonabelian groups, without proving that we have classified them all. Finally, in Section \ref{S-non-gen}, we briefly discuss some abelian groups that have no one-subsets that generate them.

\section{One-subsets of the integers}
 Before classifying the one-subsets of the integers $\mathbb Z$ in Theorem \ref{ss1} below, we give some additional background for this case.
 
Let $S$ be a nonempty set of integers. If $S$ is closed under subtraction, namely, 
\begin{equation}\label{ss10}
S-S \subseteq S,\end{equation} which means that for any $n,m \in S$
the difference $n-m$ belongs to $S$, then it is easy to see that either 
$S=\{0\}$ or $S=g \Z$ for some positive integer $g$. 
Here and below, $\Z$ stands for the set of all integers and 
$g \Z= \{gk \>|\> k \in \Z\}$. 

Indeed, assume that $S$ contains a nonzero element, say,
$g$. Then $0=g-g \in S$ and $-g=0-g \in S$. 
Hence $2g=g-(-g) \in S$ and similarly $-2g \in S$.
Continuing in this way, we deduce $g \Z \subseteq S$. 
In particular, this implies that $S$ contains a positive integer. Let $g$ be the smallest positive integer in $S$. Then  $g\Z \subseteq S$, and we claim that
$S=g\Z$.
For a contradiction, suppose that $S$ contains some element $h$, not divisible by $g$ (which can be possible only if $g>1$). As above, since $|h| \in S$, we can assume that $h>0$ and so $h>g$ due to the minimality of $g$. But then, writing  $h=kg+u$, with $k \in \N$
and $0<u<g$, from $h, kg \in S$  we derive by \eqref{ss10} that $u=h-kg \in S$, which contradicts the minimality of $g$. 

If the set $S$ is closed under addition, namely, 
\begin{equation}\label{ss5}
S+S \subseteq S,
\end{equation}
which means that
for any $n,m \in S$
the sum $n+m$ belongs to $S$, then $S$
is called an {\it additive semigroup} of integers. The structure of such sets is nontrivial and cannot be characterized in a simple way similar to the case of integer sets closed under subtraction. 

In fact, there is a considerable literature related to the additive semigroups $S \subseteq \N \cup \{0\}$ containing zero. In particular, if no integer $g>1$ divides all elements of $S$, then, by a result of Higgins \cite[Theorem 1]{hig},  the set $S$ satisfying \eqref{ss5} must contain all sufficiently large positive integers. 
Some further description of such $S$ was also independently given in \cite{dim} (see also \cite{ang}). 
Such sets are called {\it numerical semigroups}. See, for instance, the monographs \cite{assi}, \cite{how}, \cite{ros}
and the papers \cite{abu}, \cite{bras}, \cite{kap}.
The largest positive integer which does not belong to the numerical semigroup $S$ is called its {\it Frobenius number}, while the cardinality of the set $\N \setminus S$ is called its {\it genus}. 
There is a considerable literature devoted to the study of various questions related to those notions;
see, for instance, some very recent papers \cite{ar}, \cite{bog}, 
\cite{bran}, \cite{cis}, \cite{ell}, \cite{elm}, \cite{gar}, \cite{jim}, \cite{kom}, \cite{lee}, \cite{oje}.

By the above, each set of integers $S$ satisfying \eqref{ss10} automatically satisfies \eqref{ss5}.
Now, we will consider the sets of integers $S$
such that for all $n,m \in S$ (possibly $n=m$) {\it exactly one} of the numbers $n-m$, $n+m$ belongs 
to the set $S$. (Note that they are both equal if and only if $m=0$.) Recall that in this paper such $S$ is called a one-subset of $\Z$. 
It turns out that all such sets also have a very simple structure and can be fully characterized. 

\begin{theorem}\label{ss1} Let $S$ be a nonempty one-subset of the integers $\mathbb Z$.
Then one of the following is true:
\begin{itemize}
\item[$(i)$] $S=\{0\}$;
\item[$(ii)$] $S=\{0,g\}$ for some integer $g \ne 0$;
\item[$(iii)$] $S=g \Z \setminus 3g\Z$  for some $g \in \N$.
\end{itemize}
Conversely, each of the sets $S$ described in $(i), (ii), (iii)$ is a one-subset of $\mathbb Z$. 
\end{theorem}

The proof of Theorem~\ref{ss1} is elementary, but not as straightforward as in the case of integers sets satisfying \eqref{ss10}. It is given in the next section.

\section{Proof of Theorem~\ref{ss1}}

In our proof we make frequent use of the fact that for $n,m\in S$ with $m\ne 0$ exactly one of the integers $n-m$ and $n+m$ lies in $S$.

	For a given integer $g \ne 0$, we define
	\[F_g:=\{0,g\}\]
	and 
	\[I_g:=\{gk \>|\> \text{$k$ runs over all integers not divisible by 3}\}=g\Z \setminus 3g \Z.\] 
	We need to show that all possible sets satisfying the condition of the theorem are $S=\{0\}$ (case $(i)$), $S=F_g$ for some integer $g \ne 0$ (case $(ii)$), or $S=I_g$ for some integer $g \geq 1$ (case $(iii)$). Note that $I_g=I_{-g}$.
	
	The sets $S=\{0\}$ and $S=F_g$ clearly satisfy the condition of the theorem. To verify this property for $S=I_g$, where $g \geq  1$ is an integer, assume that $n,m \in I_g$. Then  
	there are integers $u,v$ not divisible by $3$ such that $n=ug$ and $m=vg$.
	Exactly one of the integers 
	$u-v, u+v$ is divisible by $3$, so exactly one of the numbers
	$n-m=(u-v)g$, $n+m=(u+v)g$ belongs to $I_g$. Thus for any   $n,m \in I_g$ the set $\{n-m,n+m\} \cap I_g$
	consists of a single element, and hence the set $I_g$ satisfies the property described in condition of the theorem too.
	
Let $S$ be a nonempty set that satisfies the required property.
Assume first that $0\in S$. 
If $S$ contains at most one other element, then $S=\{0\}$ or $S=F_g$ for some
$g \ne 0$. Each of those sets satisfies the condition of the theorem.
We will show that in the case when $0 \in S$ there are no other such sets. 

Note that if $S$ satisfies the condition of the theorem, then so does also the set $-S$. Hence, by replacing $S$ by $-S$ if necessary, we can assume that $S$ contains an element $g$ that is both positive and an element of the smallest absolute value in $S$. Since $S \ne F_g$, it must contain an element outside the set
$\{-(g-1),-(g-2),\dots,g\}$, say $h$. Without loss of generality, we may assume that $h$ has the smallest absolute value among all elements
of $S \setminus F_g$. Since $0,g \in S$ and $0+g=g \in S$ we must have $0-g=-g \notin S$, so $|h|>g$. 
 This implies that either $h>g$ or $h<-g$. Note that $0,h,0+h \in S$, so $0-h=-h \notin S$. 
\begin{itemize}
\item The case $h>g$. Then $g$ and $h$ are two smallest positive elements of $S$, so $h-g\notin S$, which yields $h+g\in S$. Now
$
g-(h+g)=-h\notin S$, 
so that $g+(h+g)=h+2g\in S$. However,
$
(h+g)-g=h\in S$,
so $(h+g)+g=h+2g \notin S$, a contradiction.
\item  The case $h<-g$. Then $h$ is the largest negative element of $S$.
 As $h<h+g<0$, we have $h+g\not\in S$, and hence $h-g\in S$. Then
$
h-(h-g)=g\in S$,
so that $h+(h-g)=2h-g\not\in S$. From
$
(h-g)-h=-g\not\in S
$
we get $(h-g)+h=2h-g\in S$, a contradiction again.
\end{itemize}
Hence $S$ containing $F_g$ cannot have a third element $h$.

In all that follows, we shall consider the case where $S$ satisfies the condition of the theorem and $0 \notin S$. 
We first note that
\begin{itemize}
	\item[(i)]  If $n\in S$, then $2n \in S$ and $3n \notin S$.
	\item[(ii)] $S=-S$.
\end{itemize}
Indeed, since $n \in S$ and $0=n-n \notin S$, we have $2n=n+n \in S$.
From $2n-n=n \in S$ it follows that $2n+n =3n \notin S$, which gives (i).  
Also, for each $n \in S$, from $n+2n \notin S$ (see (i)) we deduce $-n=n-2n \in S$, which gives (ii).

Since $S$ is nonempty, in view of $0 \notin S$ and (ii), $S$ has
a positive element. Let $g \in S$
be the smallest positive element of $S$. 
Our solution will be complete when we have shown that $S=I_g$. 

We first claim that $I_g \subseteq S$ and that no other elements in $S$ except for those in $I_g$ are divisible by $g$. To prove this,
by (ii), $0 \notin S$ and the definition of $I_g$, it suffices to show that for every integer $k \geq 1$
\[kg \in S \quad 
 \iff \quad 3 \nmid k.\] By (i) and $g \in S$, we have $g,2g \in S$ and $3g \notin S$, so the above property holds for $k=1,2,3$. Assume that $k \geq 4$ is the least integer for which the above property fails to hold. If $3 \mid k$,
then this means that $kg \in S$. But then from $g, (k-1)g \in S$ and $(k-1)g-g=(k-2)g \in S$, we obtain
$(k-1)g+g=kg \notin S$, a contradiction. If $3 \nmid k$, then this means that
$kg \notin S$. Let $l \in \{1,2\}$ be the integer for which $3 \nmid (k-l)$.
Then $3 \mid (k-2l)$. From $(k-l)g, lg \in S$ and $(k-l)g-lg=(k-2l)g \notin S$
it follows that $(k-l)g+lg=kg \in S$, a contradiction. Consequently, such 
$k \in {\mathbb N}$ does not exist. 

By the above, if $S$ contains some element different from those in $I_g$,
then this element is not divisible by $g$. Assume that such an element exists and let $h$ be the smallest such element. By (ii) and the choice of $g$, we can assume that $h>g>0$, where
$g \nmid h$.
Then $h-g>0$ cannot be in $S$ and less than $g$ (that would contradict the definition of $g$) or be in $S$ and be greater than $g$ (that would contradict the definition of $h$). Also, $h-g \ne g$, since $g \nmid h$. Hence $h-g\not\in S$ and therefore $h+g\in S$. Then
$
(h+g)-g=h\in S$,
so $(h+g)+g=h+2g\not\in S$. Since $h,2g\in S$, this implies
$h-2g\in S$.
Now $h-2g \notin \{-g,0,g\}$ and $h-2g>g-2g=-g$. If $h-2g>g$, then $h-2g \in S$ contradicts the definition of $h$, while if $g>h-2g>-g$, then $|h-2g| \in S$ contradicts the definition of $g$. Hence such an $h$ does not exist, and therefore $S=I_g$.

\section{Generalisation to abelian groups}

In an additive abelian group $G$ containing one-subset $S$,
from the definition of one-subset we have two possibilities: for $g,h\in S$ either $2h\ne 0$ and exactly one of $g-h$ and $g+h$ is in $S$, or $2h=0$ and $g-h=g+h\in S$. For example, this latter possibility occurs when $h=0$.

 We also define an {\it involution subgroup}
$I$ of $G$ to be a subgroup of $G$ such that all nonzero elements of $I$ are involutions (elements that are their own inverses). For example, for $G=\Z$, there are no nonzero involutions, so the only possibility is $I=\{0\}$.

\subsection{One-subsets that contain $0$ in abelian groups } 

We first describe all possible one-subsets $S$, where $0 \in S$, of an abelian group $G$. 

\begin{theorem}\label{Inv} Let $G$ be an abelian group with $S$ a one-subset of $G$ containing $0$. Then one of the following is true:
\begin{itemize}
\item[$(i)$] $S$ is an involution subgroup $I$ of $G$;
\item[$(ii)$] $S=I\cup (g+I)$, where $I$ is an involution subgroup of $G$, with $g\in G$ such that $2g\not\in I$.
\end{itemize}
Conversely, each of the sets $S$ described in $(i), (ii)$ is a one-subset of $G$. 
\end{theorem}

\begin{proof} Let $I$ be an involution subgroup of $G$. 
From $i+h=i-h=h-i \in I$ for $i,h\in I$, we see that 
$S=I$ is a one-subset of $G$. 
To prove that $S$ as in $(ii)$ is a one-subset of $G$, take any involution subgroup $I$ of $G$ and any $g\in G$ with $2g\not\in I$. Then $g\not\in I$, since $I$ is a group. 
We want to show that $I \cup (g+I)$ is a one-subset of $G$.
Consider the elements $g+i$ and $g+i'$ in $g+I$. Then $(g+i)-(g+i')=i-i'\in I\subset S$,
but $(g+i)+(g+i')=2g+i+i'\not\in I$, as otherwise $(2g+i+i')+(i+i')=2g$ would be in $I$, which it isn't. Hence $(g+i)+(g+i')\not\in S$, since it does not belong to $g+I$ either. Also note that if $i\in I$ and $g+i'\in g+I$ then $(g+i')+i=(g+i')-i \in g+I$, and $i-(g+i')\not\in I$, for otherwise $i-(g+i')-i+i'=-g$ would also be in $I$, which it isn't.  Hence $i-(g+i')\not\in S$, since it does not belong to $g+I$ either. Thus 
$S=I\cup (g+I)$ is a one-subset of $G$.

It remains to show that no other options are possible. 
 Since $0\in S$ by assumption,  
the set of involutions $I$ of $S$ is nonempty. 
Furthermore, since for any $i,j \in I \subseteq S$ we have
$i+j=i-j=j-i \in S$ and $i+j$ is an involution thanks
 to $i+j=-i-j$, the set $I$ is a subgroup of $G$. 
There is nothing to prove if $S=I$, so
 suppose that a one-subset $S$ of $G$ contains 
 an element $g \in S \setminus I$.  Then $0=g-g \in I \subseteq S$, so
 $2g=g+g$ is either $0$ or not in $S$. However, $2g \ne 0$, since $g$ is not an involution, which yields
$2g \notin S$. (In particular, this implies $2g \notin I$.)
From $0+g=g\in S$ we have $0-g=-g\not\in S$, as $-g\ne g$. 
Hence $-g'\not\in S$ for any $g'\in S\setminus I$. 

Assume that $S \setminus (I \cup \{g\})$ contains another element $h$. 
If $h-g\not\in I$,  then one of $g-h$ and $h-g$ is not in $S$, so that $g+h\in S$. But then $g-(g+h)=-h\not\in S$, so that $g+(g+h)=2g+h\in S$,
 and $(g+h)-g=h\in S$, giving $(g+h)+g=2g+h\not\in S$, a contradiction.
  Hence $h-g\in I$, so that $h\in g+I$.
  Consequently, for any $g \in S \setminus I$, we must have $I \cup \{g\} \subseteq S \subseteq I\cup (g+I)$.
  But for each $i \in I$ we have
  $g+i=g-i$, so $g+i \in S$, and hence $S=I\cup (g+I)$.
\end{proof}

Theorem~\ref{Inv} implies the following corollary:

\begin{corollary}\label{noinv} If an abelian group $G$ contains no involutions except for $0$ then its only one-subsets containing $0$ are $\{0\}$ and $\{0,g\}$ for any nonzero $g\in G$. At the other extreme, if all nonzero elements of $G$ are involutions, then all subgroups of $G$ are one-subsets.
\end{corollary}

In the next subsection, we will describe one-subsets $S$ of an abelian group $G$ with the property $0 \notin S$. 

\subsection{Properties of one-subsets not containing $0$ in an abelian group}

\begin{lemma}\label{L-trans} Suppose that $S$ is a one-subset of some abelian group $A$,  and $0\not\in S$. Suppose that $a,b,c,a+b$ and $b+c$ are all in $S$. Then $a+c\in S$. \end{lemma}
\begin{proof}Suppose too that $a+c\not\in S$. Then $a-c\in S$. Also either $a+b+c$ or $a+b-c\in S$. 

But if $a+b+c\in S$ then $a+2b+c=(a+b+c)+b\in S$, as $(a+b+c)-b=a+c\not\in S$. 
But also $a+2b+c=(a+b)+(b+c)\not\in S$ as $(a+b)-(b+c)=a-c\in S$. A contradiction.

Also if $a+b-c\in S$, then  $(a+b)+c\not\in S$. Then $b-a+c\in S$ as it equals $(b+c)-a$, and $(b+c)+a=a+b+c\not\in S$. 
On the other hand, $b-a+c=b-(a-c)\not\in S$ as $b+(a-c)\in S$. Another contradiction. So $a+c\in S$.
\end{proof}

\begin{proposition}\label{P-props}In an abelian group $A$, let $S$ be a one-subset not containing $0$. Then $S$ has the following properties:
\begin{itemize}
\item[$(i)$] If $s\in S$ then $2s\in S$.
\item [$(ii)$] If $s\in S$ then $3s\not\in S$.
\item [$(iii)$] $s\in S$ if and only if $-s\in S$.
\item [$(iv)$] For $s$ and $t$ in $S$ write $s\sim t$ if $s+t\in S$. Then $\sim$ defines an equivalence relation on $S$.
\item[$(v)$] $S$ has exactly two equivalence classes $S_+$ and $S_-$ under $\sim$. For every $s\in S$, the elements $s$ and $-s$ lie in different classes. Thus for every $t\in S$ either $t\sim s$ or $t\sim -s$.
\item[$(vi)$] If $a,b$ and $c$ belong to the same equivalence class, then $a+b+2c$ lies in $S$ and belongs to the same equivalence class.
\item[$(vii)$] If $a$ and $b$  belong to the same equivalence class then $a+b\sim -a$.
\item[$(ix)$] If $a,b$ and $c$ belong to the same equivalence class, then $a+b+c\not\in S$.
\item[$(x)$] If $a,b,c$ and $d$ belong to the same equivalence class,
then $a+b+c+d$ also belongs to this equivalence class, and so certainly belongs to $S$. However  $a+b-c-d\not\in S$.
\end{itemize}
\end{proposition}
\begin{proof}
 Let $s,t,a,b,c,d\in S$.
 \begin{itemize}
\item[$(i)$] 
We have $s-s=0\not\in S$, so $s+s=2s\in S$.
\item [$(ii)$] 
We have $2s-s=s\in S$, so $2s+s=3s\not\in S$.
\item [$(iii)$] 
We have $s+2s=3s\not\in S$, so $s-2s=-s\in S$.
\item [$(iv)$] 
For $a,b,c\in S$ we have $a\sim a$ by $(i)$,  and if $a\sim b$ then $a+b=b+a\in S$, so $b\sim a$. Finally, the equivalence relation $\sim$ is transitive on $S$ by Lemma \ref{L-trans}.
\item[$(v)$] 
Because $s+(-s)=0\not\in S$, $s$ and $-s$ are not equivalent under $\sim$. Furthermore, either $t+s\in S$ or $t-s=t+(-s)\in S$, so $t\sim s$ or $t\sim-s$, so that the only two equivalence classes are those containing $s$ and $-s$.
\item[$(vi)$] 
Now $(b+c)-(a+c)=b-a\not\in S$, so that $(b+c)+(a+c)=a+b+2c\in S$.
\item[$(vii)$] 
This follows because $(a+b)+(-a)=b\in S$.
\item[$(ix)$] 
If $a+b+c\in S$ then from $(vi)$ we have $a+b+2c=(a+b+c)+c\in S$, so that $(a+b+c)-c=a+b\not\in S$, a contradiction.
\item[$(x)$]
From $(vii)$, $a+b$ and $c+d$ belong to the same equivalence class (namely, the one not containing $a,b,c,d$), so, again by $(vii)$, their sum $(a+b)+(c+d)$ belongs to the equivalence class containing $a,b,c,d$. Thus their difference $(a+b)-(c+d)$ does not belong to $S$.
\end{itemize}
\end{proof}

\begin{theorem}\label{T-nozero} Let $G$ be an abelian group having a subgroup $H$ of index $3$, and let $a$ be an element of $G\setminus H$. 
Then the subset $\{a+H\}\cup\{-a+H\}$ of $G$ is a one-subset which does not contain $0$. Furthermore, this one-subset generates $G$.
 \end{theorem}

 \begin{proof}
Let $S$ be the union $\{a+H\}\cup\{-a+H\}$ of two cosets of $H$ in $G$, and let $s$ and $s'$, not necessarily distinct, lie in $S$. Note that $0\not\in S$ and $3a\in H$. If $s$ and $s'$ both lie in the same coset, their sum lies in the other coset, and hence in $S$, while their difference lies in $H$, so is not in $S$. On the other hand, if $s$ and $s'$ lie in different cosets, then their sum lies in $H$, so not in $S$, while both their differences $s-s'$ and $s'-s$ lie in $S$. 
Thus $S$ is a one-subset of $G$. The final sentence of the theorem  follows easily.
\end{proof}

 Immediately, we have the following more general result.
\begin{corollary}\label{C-nozero}  If $A$ is any abelian group, then, for all subgroup pairs $H\subset G$ of $A$  with $[G:H]=3$ and any $a\in G\setminus H$,  the subsets $\{a+H\}\cup\{-a+H\}$ of $G$ are one-subsets of $A$ that do not contain $0$.
\end{corollary}

We note that Theorem \ref{ss1}, applied to $A=\mathbb Z$, is stronger than this corollary, since it tells us that {\it all} one-subsets of $\mathbb Z$ not containing $0$ are produced in this way. Now, we prove in Theorem \ref{T-Abmain} that this also holds more generally.

\begin{theorem}\label{T-Abmain}Let $A$ be an abelian group containing a one-subset $S$ that does not contain $0$. Then inside $A$ the set $S$ generates a subgroup $G$ that itself contains a subgroup $H$ of index $3$. On partitioning $G$ into its 3 cosets of $H$, namely $H$, $a+H$ and $-a+H$, for some $a\in G\setminus H$, we have that $S=(a+H)\cup(-a+H)$. Furthermore, the subgroup $H$ is given by
\[
H=S_{+}-S_{+}=S_{-}-S_{-}.
\]
\end{theorem}
\begin{proof} By Proposition \ref{P-props}$(iii)$ we know that if $s\in S$ then $-s\in S$. 
Thus every element $g$ of $G$ can be written as a sum of several elements of $S$. 
Since the sets $S_{+}$ and $S_{-}$ are disjoint, we can write $g$ as the sum of several summands from $S_{+}$
and several summands from $S_{-}$. 
From Proposition \ref{P-props}$(vii)$ it follows that any two 
summands from
$S_{+}$ (or resp. from $S_{-}$) can be replaced by a single summand from $S_{-}$ (resp. $S_{+}$). 
In this way, we can reduce the number of summands in each sum to $0$ or $1$. Hence either $g=b$ for some $b\in S_+$ or $g=-c\in S_-$ (which means that $c\in S_+$) or $g=b-c$ for some $b$ and $c$ in $S_+$. Thus 
$g$ is either in $S$ or in the difference set $S_{+}-S_{+}$. 
Put $H=G\setminus S=S_{+}-S_{+}$. By the identity
$b-c=-c-(-b)$
it is clear that
$S_{+}-S_{+}=S_{-}-S_{-}$.

To show that $H$ is a subgroup of $G$, pick any $b\in S_+$. Then $0=b-b\in H$. Also, if $h\in H$ is given for some $b,c$ in $S_+$ by $h=b-c$, then $-h=c-b\in H$. Finally, if $h=b-c$ and $h'=b'-c'$ for $b,c,b',c'\in S_+$ then
$h+h'=b+b'-(c+c')= (-(c+c'))-(-(b+b'))$, where $c+c'$ and $b+b'$ lie in $S_-$, so by Proposition \ref{P-props}$(v),(vii)$ we see that $-(c+c')$ and $-(b+b')$ lie in $S_+$. Hence $h+h'\in H$, so that $H$ is a subgroup of $G$.

Now pick any $a\in S_+$. Then for any $b\in S_+$ we have $b=a+(b-a)\in a+H$, while for any $-c\in S_-$ we have $-c =-a+(a-c)\in -a+H$ because $c \in S_{+}$. Hence $S_+=a+H$ and $S_-=-a+H$. From Proposition~\ref{P-props}$(ii)$ we deduce $3a\in H$, so that $[G:H]=3$.
\end{proof}

\begin{corollary}\label{C-3}  For any $a\in G\setminus S$, subgroup $H$ of $G$ as in the theorem, and every integer $k\ge 0$
\begin{itemize}
 \item[$(i)$] the sum of $3k$ elements of $a+H$ lies in $H$;
 \item[$(ii)$] the sum of $3k+1$ elements of $a+H$ lies in $a+H$;
  \item[$(iii)$] the sum of $3k+2$ elements of $a+H$ lies in $-a+H$.
\end{itemize}
\end{corollary}

\begin{proof}
 Clearly, $(i)$ is true for $k=0$ (empty sum) and for $k=1$ by Proposition \ref{P-props}$(ix)$.
 Assume $(i)$ is true for $k$, where $k\ge 1$. Then the sum of $3(k+1)=(3k-1)+4$ elements of $a+H$ is, by Proposition \ref{P-props}$(x)$, a sum of $(3k-1)+1=3k$ elements of $a+H$, and so is in $a+H$ by the induction hypothesis.
 
 Next, $(ii)$ is trivially true for $k=0$. Then proceed as in $(i)$.
Finally, $(iii)$ is true for $k=0$ by Proposition \ref{P-props}$(vii)$. Then proceed as in $(i)$.
\end{proof}

\section{Examples}\label{S-ex}
In this section we describe how to find all one-subsets of particular abelian groups $G$.

\subsection{Preliminary lemma}

\begin{lemma}\label{L-indexp}
Let $p$ be a prime number and $A$ be a finite abelian $p$-group. Then $A$ has $(p^k-1)/(p-1)$ subgroups of index $p$, where $A$ is a direct sum of $k$  $p$-power cyclic groups $C_j$ with generators $g_j\,(j=1,\dots, k)$. These subgroups can be described as follows:  let $f:A\to \Z/p\Z$ be a homomorphism that maps $g_j \mapsto r_j\,(j=1,\dots, k)$. Further specify that not all $r_j$ are $0$, and the least $j$ for which $r_j\not= 0$ has $r_j=1$. Then the subgroups of $A$ of index $p$ are the kernels of such maps, which are distinct.
\end{lemma}

\begin{proof} Since all subgroups of $A$ are normal, each one of index $p$ is the kernel of some surjective homomorphism $f:A\to \Z/p\Z$. 
 Let $f$ be the map described in the lemma, so that for integers $m_j$ we have $\sum_{j=1}^k m_jg_j \stackrel{f}{\mapsto}  \sum_{j=1}^k m_jr_j$. This is surjective unless  all $r_j=0$. So assume $f\ne 0$. Then 
\[
\ker f= \left\{\sum_{j=1}^k m_jg_j: \sum_{j=1}^k m_jr_j=0\right\},
\]
so that $\ker f$ is determined by $\underline{r}:=
(r_1,\dots,r_k)$. Furthermore two vectors $\underline{r}$ and $\underline{r'}\in (\Z/p\Z)^k$ clearly determine the same subgroup of $A$ if and only if $\underline{r}=\lambda\underline{r}$
for some $\lambda\in (\Z/p\Z)^\times$. Thus we have a bijection between subgroups of index $p$ in $A$ and vectors $\underline{r}\in (\Z/p\Z)^k$ whose least $j$ for which $r_j\not= 0$ has $r_j=1$.
\end{proof}

\begin{corollary}\label{C-pgp} Given a finite abelian group of order divisible by a prime $p$, written as $A\oplus A'$, where $A$ is a $p$-group and $A'$ has order coprime to $p$, then the subgroups of $A\oplus A'$ of index $p$ are of the form $H\oplus A'$, where $H$ is a subgroup of $A$ of index $p$, as given by Lemma \ref{L-indexp}.
\end{corollary}
This is clear from the fact that any homomorphism from $A\oplus A'$ to $\Z/p\Z$ must map $A'$ to $0$.
\subsection{Finite abelian groups}

 Let $G$ be a finite additive abelian group. Here we apply Theorem \ref{Inv} to first describe all its one-subsets that contain $0$. Let us write $G$ as a direct sum 
\[
G=E_1\oplus\cdots\oplus E_m\oplus A'
\]
of groups, where all $E_i:=\langle e_i \rangle$ have orders $2n_i$ that are powers of $2$, and $A'$ has odd order. Every involution subgroup $I$ of $G$ is a subgroup of
\[
F:=\langle f_1 \rangle\oplus\cdots\oplus\langle f_m \rangle,
\]
where $f_i=n_ie_i$. Let us abbreviate a typical element $\sum_i b_if_i$ of $F$ by $(b_1,\dots,b_m)$, where all $b_i\in \{0,1\}$. Every involution subgroup $I$ of $G$ will have a basis $\pmod 2$ of such vectors, which we can assume are in echelon form. This gives all possible involution subgroups of $G$ which, by Theorem \ref{Inv}(i), are one-subsets of $G$ that contain $0$. To find all other one-subsets of $G$ that contain $0$ we need to find, by Theorem \ref{Inv}(ii), all $g\in G$ with $2g\not\in I$.  To do this, we find (and eliminate) all $g\in G$ with $2g\in I$. 
Clearly the component of such a $2g$, and thus also the component of $g$
in $A'$, must be $0$. Also, the components of $2g$ in $E_1\oplus\cdots\oplus E_m$ must lie in $F$ and the corresponding binary vector must lie in the span of the binary vectors that form a basis $\pmod 2$ for $I$. Put
\[
2g=b_1'f_1+\cdots+b_m'f_m=b_1'n_1e_1+\cdots+b_m'=b_1'n_1e_1+\cdots+b_m'n_me_m
\]
and $g=n_1'e_1+\cdots+n_m'e_m$, say. If $b_i'=0$ then $2n_i\mid 2n_i'$, so $n_i'=0$ or $n_i$, while if $b_i'=1$ then $2\mid n_i$ and $(n_i/2)\mid n_i'$. Thus we can find all the  $g$  forbidden by Theorem \ref{Inv}(ii).

We now find the one-subsets of $G$ that do not contain $0$. 
Since a one-subset of $G$ is also a one-subset of any abelian group containing $G$, we can confine our attention to describing the one-subsets of this type that generate $G$. Thus any such one-subset of a proper subgroup of $G$ will also be a one-subset of $G$. From Theorem \ref{T-Abmain} we know that these one-subsets are described using the subgroups $H$ of $G$ of index $3$. Thus we need to find such subgroups $H$, which are the kernels of all possible surjective homomorphisms from $G$ to the cyclic group $C_3$ of order $3$. These are described in Lemma \ref{L-indexp} and Corollary \ref{C-pgp} for $p=3$.

\begin{corollary}\label{C-finite} Let $C_n=\{0,1,\dots,n-1\}$ be a finite additive cyclic group of order $n$, with $U_n$ denoting its subset of elements that are coprime to $3$. Then a complete list of the one-subsets of $C_n$ is the following:
\begin{itemize}
\item[$(i)$] $\{0\}$;
\item[$(ii)$] $\{0,g\}$ for some $g\in C_n$ with $g \ne 0$;
\item[$(iii)$] $\{0,g,\frac{n}2,g+\frac{n}2\}$, where $n \geq 6$ is even, $1\le g<\frac{n}2$ and $g\ne \frac{n}4$;
\item[$(iv)$] $gU_{n/g}$ for some $n$ divisible by $3$ and some $g$ dividing $\frac{n}{3}$.
\end{itemize}
\end{corollary}

The proof of parts $(i)-(iii)$ of the corollary come straight from Theorem \ref{Inv}. Part $(iv)$ comes from Theorem \ref{T-Abmain}, using the fact that the subgroups of $C_n$ containing a (sub)subgroup of index $3$ are the subgroups $C_{n/g}$, where $g\in C_n$ divides $n/3$.

Note that possibilities $(i), (ii), (iv)$
are similar to those described in Theorem~\ref{ss1}, while
$(iii)$ is different and reflects the fact that $C_n$ is finite. 

\subsection{The lattice $\Z^n$}\label{SS-Zn}
Since $\Z^n$ contains no nontrivial involutions, we know from Theorem \ref{Inv} that its only one-subsets containing $\underline{0}$ are $\{\underline{0}\}$ and $\{\underline{0},\underline{v}\}$ for any $\underline{v}\ne\underline{0}\in \Z^n$.

To obtain the one-subsets $S$ of $\Z^n$ not containing $\underline{0}$, we first assume that 
the $\Z$-span of $S$ is the whole of $\Z^n$. We apply Theorem \ref{T-Abmain} with $G=A=\Z^n$.
 Thus $S$ is of the form $(\underline{a}+H)\cup(-\underline{a}+H)$ for some sublattice $H$ of index $3$ in $\Z^n$, and $\underline{a}\in \Z^n\setminus H$. We now choose an $n\times n$ integer matrix of determinant $3$, $A_H$ say, whose columns are a $\Z$-basis of $H$. Thus by putting $A_H$ in (column) Hermite normal form we can assume that 
 \begin{equation*}
 A_H=\left(
 \begin{matrix}
 1 & 0 &  0 & \cdots &  &  & \cdots &  &  0 \\
 0 & 1 & 0 & \cdots &   &  &  &  &  0 \\
 \vdots   &   &   \ddots  & \ddots  &  &  &  &  & \vdots\\
0 & 0 & \cdots & 1 & 0 &  0 &  & \cdots & 0  \\
a_1 & a_2 & \cdots & a_{j-1} & 3 & 0 & & \cdots & 0 \\
0 & 0 & 0 &  \cdots & 0 & 1 &  & \cdots & 0 \\
0 & 0 & 0 &  \cdots & 0 & \ddots & \ddots & \vdots &0 \\
 \vdots   &   & &  & &  & \ddots & \ddots & 0\\
0 & 0 & 0 &  \cdots & 0 & 0 & \cdots &  0 & 1 
 \end{matrix}
 \right)
 \end{equation*}
where $a_1,a_2,\dots,a_{j-i}\in \{0,1,2\}$. Each matrix $A_H$ corresponds to the one-subset 
\[
S=(\underline{a}+H)\cup(-\underline{a}+H),
\]
where we can take $\underline{a}$ to be the the vector with $1$ in its $j$-th place and $0$ elsewhere. 

As a check that $H$ has index $3$ in $\Z^n$, we note that the columns of the matrix $A_H'$, defined as $A_H$ emended by replacing its $j$-th column (which contains $3$) by $\underline{a}$,
span $\Z^n$. Thus $\Z^n$ is indeed the disjoint union of $H$, $\underline{a}+H$ and
$(-\underline{a}+H)$. Note too that $S$ spans $\Z^n$.

We now complete our search for one-subsets $S$ of $\Z^n$ not containing $\underline{0}$, where we now allow the $\Z$-span of $S$ to be any (additive) subgroup $G$ of $\Z^n$. So $G$ is a $k$-dimensional lattice dimension $k$, where $1\le k\le n$. Such a lattice is isomorphic to $\Z^k$,
so the method above, with $n$ replaced by $k$ gives the construction of all sublattices $H$ of index $3$ in $G$.  To take the very simple case $k=1$, the $\Z$-span of $S$ is the group of integer multiples of some $g\in\Z$. Its only subgroup of index $3$ is the group $H$ generated by $3g$, so that $S=(1+H)\cup(-1+H)= \{j\in \Z: 3\nmid j\}$, as in Theorem \ref{ss1}$(iii)$.

\subsection{The $3$-adic integers $\Z_3$}
Recall that for the field $\Q_3$ of $3$-adic numbers, armed with the $3$-adic valuation $|\,\,\,|_3$, its ring of integers $\Z_3$ is $\{x\in \Q_3: |x|_3\le 1\}$. In this ring, the set $\{x\in \Q_3: |x|_3 < 1\}$ is a maximal ideal. Denoting this ideal by $H_3$, and considering the additive structure only, we see that $S:=(1+H_3)\cup(-1+H_3)$ is a one-subset of $\Z_3$ that does not contain $0$. 

Further one-subsets $S_F$ of $\Z_3$ can be obtained using subfields $F$ of $\Q_3$, by intersecting $H_3$, and then defining $S_F:=(1+H_3\cap F)\cup(-1+H_3\cap F)$. This construction certainly does not exhaust the one-subsets of $\Z_3$ that do not contain $0$. For instance, the $\Z$-span of any set of $n$ elements of $\Z_3$ that are linearly independent over $\Z$ is of course isomorphic as an additive group to $\Z^n$, so additional one-subsets of $\Z_3$ can be constructed using the methods of
subsection \ref{SS-Zn} above.

\section{Nonabelian groups}

We now extend the notion of a one-subset to nonabelian groups, which we write multiplicatively. More precisely, given a group $G$ with group operation $*$, we are searching for subsets $S$ of $G$ with the property that for any elements $i,j\in S$ we have 
\begin{equation}
\label{byu}    
\#(\{i*j,i*j^{-1}\} \cap S)=1.
\end{equation}
We generalise from the abelian case to also call such a set $S$ a {\it one-subset} of $G$.

\begin{theorem}\label{T-Gnozero} Let $G$ be a group with identity element $e$  having a normal subgroup $H$ of index $3$, and let $a$ be an element of $G\setminus H$. 
Then the subset $\{aH\}\cup\{a^{-1}H\}$ of $G$ is a one-subset which does not contain $e$.
 \end{theorem}

 \begin{proof}
Let $S$ be the union $\{aH\}\cup\{a^{-1}H\}$ of two cosets of $H$ in $G$, and let $s$ and $s'$, not necessarily distinct, lie in $S$. Note that $e\not\in S$ and $a*a*a\in H$. If $s$ and $s'$ both lie in the same coset, then $s=a^\eps * h$ and  $s=a^\eps * h'$ say, where $\eps=\pm 1$. Their product
is
\[
s*s'=a^\eps *h*a^\eps *h'= a^{2\eps}*(a^{-\eps}*h*a^\eps) *h'=a^{2\eps}*h''*h'=a^{-\eps}*h'''
\]
for some $h'',h'''\in H$, so that it lies in the other coset, and hence in $S$. 
(Here and below, $a^{2\eps}=a^{\eps}*a^{\eps}$.)
However,  
\[
s*(s')^{-1}=a^\eps * h*(h')^{-1}*a^{-\eps} \in H,
\]
 so is not in $S$. On the other hand, if $s$ and $s'$ lie in different cosets, say $s=a^{\eps}*h$ and $s'=a^{-\eps}*h'$, then 
 \[
s*s'=a^\eps *h*a^{-\eps} *h'
 \]
 is of the form $h''h'$, with $h''=a^{\eps}*h*a^{-\eps} \in H$, so it lies in $H$, while
 \[
s*(s')^{-1}=a^\eps *h*(h')^{-1}*a^{\eps}= a^{2\eps}*(a^{-\eps}*h*(h')^{-1}*a^{\eps}),
\]
 which lies in $a^{2\eps}H=a^{-\eps}H$.
 Therefore $S$ is a one-subset of $G$.
\end{proof}

 Immediately, we have the following more general result.
\begin{corollary}\label{C-Gnozero}  If $B$ is any  group then for all subgroup pairs $H\subset G$ of $B$  with $[G:H]=3$ and $a\in G\setminus H$,  the subsets $\{aH\}\cup\{a^{-1}H\}$ of $G$ are one-subsets of $B$ that do not contain $e$.
\end{corollary}

Concerning one-subsets not containing $e$ of a nonabelian group $G$, it is trivial to see that both $\{e\}$ and $\{e,g\}$ for every $g\in G$ are one-subsets, just as in the abelian case.
 However, we do not have a complete picture of all one-subsets $S$ in the nonabelian case,  for either of the cases $e\in S$ and $e\not\in S$.

\section{Groups not generated by a one-subset} \label{S-non-gen} Both in Corollary \ref{C-nozero} and in the examples of Section \ref{S-ex} we have seen that groups inherit one-subsets from their subgroups. Thus to find all one-subsets of a group we can confine our attention to its subgroups that are generated by one of their one-subsets. In particular, divisible groups have no subgroups of finite index, and so Theorem \ref{T-Abmain}
tells us that such groups have no one-subsets  not containing $0$ that generate them. Furthermore, Theorem \ref{Inv} readily shows the same non-existence result concerning divisible groups with a one-subset containing $0$. In particular, neither the additive group of the rationals $\Q$ nor the reals $\R$ contain one-subsets that generate them.
(Both groups do however contain one-subsets. For instance, the group $\Q\cap \Z_3$, a subset of both $\Q$ and $\R$, contains a one-subset that generates
it, using the construction of Theorem \ref{T-nozero}. Here, the subgroup $H$ used for that construction is $\{x\in \Q: |x|_3 < 1\}$.


\begin{thebibliography}{00}


\bibitem{abu}
 {\sc N.~Abu-Ghazalh and N.~Ru\v skuc,}
     On disjoint unions of finitely many copies of the free
monogenic semigroup,
   {\it Semigroup Forum}
  {\bf 87}, no.~1 
      (2013),
    243--256. 

\bibitem{ang}
{\sc V.~Angjelkoska and D.~Dimovski,}
     Additive semigroups of integers. Embedding dimension of
              numerical semigroups,
{\it Pril. Odd. Prir.-Mat. Biotekhnichki Nauki}
{\bf 41}, no.~1
      (2020),
49--55. 

\bibitem{ar}
  {\sc F.~Arias, J.~Borja and C.~Rhenals,}
     The Frobenius problem for numerical semigroups generated by
              sequences of the form $ca^n-d$,
   {\it Semigroup Forum}
  {\bf 107}, no.~3
      (2023),
    581--608. 

\bibitem{assi}
    {\sc A.~Assi, M.~D'Anna and P.~A.~Garc\'ia-S\'anchez,}
     {\it Numerical semigroups and applications,}
    RSME Springer Series,
    vol. 3,
   Second ed., 
 Springer, Cham,
    2020. 

\bibitem{bog}
     {\sc T.~Bogart, C.~O'Neill and K.~Woods,}
     When is a numerical semigroup a quotient?,
   {\it Bull. Aust. Math. Soc.}
  {\bf 109}, no.~1
      (2024),
    67--76. 

\bibitem{bran}
    {\sc M.~B.~Branco, I.~Ojeda and J.~C.~Rosales,}
     Arithmetic varieties of numerical semigroups,
   {\it Results Math.} {\bf 79}, no.~4
      (2024),
    Paper No. 171, 17~p.

\bibitem{bras}
 {\sc M.~Bras-Amor\'os,}
    Numerical semigroups and codes,
 in: {\it Algebraic geometry modeling in information theory},
Ser. Coding Theory Cryptol.,
    vol. 8,  World Sci. Publ., Hackensack, NJ,
      2013,
     pp.~167--218.

\bibitem{cis}
    {\sc C.~Cisto, G.~Failla and F.~Navarra,}
     Generalized numerical semigroups up to isomorphism,
   {\it Comm. Algebra}
  {\bf 53}, no.~11
      (2025),
    4546--4564. 
     

\bibitem{dim}
    {\sc D.~Dimovski,}
    Additive semigroups of integers,
   {\it Makedon. Akad. Nauk. Umet. Oddel. Prirod.-Mat. Nauk. Prilozi}
 {\bf 9}, no.~2 
      (1977),
    21--26. 


\bibitem{ell}
  {\sc S.~Eliahou,}
     Divsets, numerical semigroups and Wilf's conjecture,
   {\it Comm. Algebra}
{\bf 53}, no.~5
      (2025),
2025--2048. 


\bibitem{elm}
     {\sc C.~Elmacioglu, K.~Hilmer,  C.~O'Neill, M.~Okandan and H.~Park-Kaufmann,}
     On the cardinality of minimal presentations of numerical
              semigroups,
   {\it Algebr. Comb.} {\bf 7},
   no.~3
      (2024),
    753--771. 

\bibitem{gar}
     {\sc P.~A.~Garc\'ia-S\'anchez,}
     The isomorphism problem for ideal class monoids of numerical
              semigroups,
   {\it Semigroup Forum}
  {\bf 108}, no.~2 
      (2024),
    365--376. 


\bibitem{hal}
{\sc Brian C. Hall,} 
{\it Lie groups, Lie algebras, and representations.
An elementary introduction.}
2nd edn.
Graduate Texts in Mathematics, 222. Springer, Cham,  2015. 



    

\bibitem{hig}
{\sc J.~C.~Higgins,}
     Subsemigroups of the additive positive integers,
   {\it Fibonacci Quart.}
  {\bf 10}, no.~3 
      (1972),
    225--230. 

\bibitem{how}
    {\sc J.~M.~Howie,}
    {\it Fundamentals of semigroup theory},
London Mathematical Society Monographs. 
Oxford University Press, New York,  1995. 


\bibitem{jim}
{\sc J.~Jim\'enez--Urroz and J.~M.~Tornero,}
    On the computation of the MED closure of a numerical semigroup,
   {\it Semigroup Forum}
  {\bf 111}, no.~1
    (2025),
    144--162.


\bibitem{kap}
      {\sc N.~Kaplan,}
     Counting numerical semigroups,
   {\it Amer. Math. Monthly}
  {\bf 124}, no.~9
      (2017),
    862--875. 

\bibitem{kom}
      {\sc T.~Komatsu and J.~Mu,}
     $p$-Numerical semigroups of Pell triples,
   {\it J. Ramanujan Math. Soc.}
    {\bf 40}, no.~1
    (2025),
    5--22.

\bibitem{lee}
     {\sc K.~Lee and H.~Nam,}
     The genus of a quotient of several types of numerical semigroups,
   {\it Int. J. Number Theory} {\bf 20}, no.~7
      (2024),
    1809--1831. 

\bibitem{oje}
 {\sc I.~Ojeda and J.~C.~Rosales,}
     The multiples of a numerical semigroup,
   {\it Turkish J. Math.}
  {\bf 48}, no.~6
      (2024),
    Art. 5, 1055--1066. 


\bibitem{ros}
 {\sc J.~C.~Rosales and P.~A.~Garc\'ia-S\'anchez,}
     {\it Numerical semigroups},
    De\-ve\-lopments in Mathematics,
    Vol.~20,
Springer, New York,
      2009.

   	\end{thebibliography}
\end{document}